\documentclass[11pt,oneside]{amsart}
\usepackage{amsmath}
\usepackage{amsthm}
\usepackage{amsfonts}
\usepackage{amssymb,amscd,epsf,verbatim}
\usepackage{mathrsfs}
\usepackage{graphicx}
\usepackage{latexsym}
\usepackage{standalone}
\usepackage{lscape}
\usepackage[colorlinks=true]{hyperref}
\hypersetup{colorlinks, citecolor=blue, filecolor=black, linkcolor=red, urlcolor=green}
\usepackage{epstopdf}
\usepackage{tikz}
\usetikzlibrary{calc}
\usetikzlibrary{matrix,arrows,decorations.pathmorphing}
\usepackage{tikz-cd}
\usepackage{color}
\usepackage{geometry}
\usepackage{multirow}
\usepackage{enumitem}
\usepackage{framed}

\newcommand{\Q}{\mathbb{Q}}
\newcommand{\Qp}{\Q_p}
\newcommand{\C}{\mathbb{C}}
\newcommand{\R}{\mathbb{R}}
\newcommand{\Z}{\mathbb{Z}}

\newcommand{\M}{\mathcal{M}}
\newcommand{\Zp}{\Z_p}
\newcommand{\F}{\mathbb{F}}
\newcommand{\Fp}{\F_p}
\newcommand{\Fq}{\F_q}

\renewcommand{\O}{\mathfrak{o}}

\newcommand{\Der}{\operatorname{Der}}
\newcommand{\Fra}{\operatorname{Frac}}
\newcommand{\Reg}{\operatorname{Reg}}
\newcommand{\hgt}{\operatorname{ht}}
\newcommand{\rank}{\operatorname{rank}}

\newcommand{\Spec}{\operatorname{Spec}}

\newcommand{\Spa}{\operatorname{Spa}}
\newcommand{\ra}{\rightarrow}
\newcommand{\Ar}{{A^r_{L,E}}}
\newcommand{\WLE}{{W(\O_L)_E}}
\newcommand{\Bi}{B^I_{L,E}}
\newcommand{\Bic}{B^{I,\circ}_{L,E}}
\newcommand{\Bip}{B^{I,+}_{L,E}}

\newcommand{\Bin}{\Bi\{T_1,\dots,T_n\} }
\newcommand{\imax}{\mathfrak{m}}
\newcommand{\ip}{\mathfrak{p}}
\newcommand{\iq}{\mathfrak{q}}
\newcommand{\ATn}{A\{T_1/\rho_1,\dots,T_n/\rho_n\}}
\newcommand{\ATrn}{\Ar\{T_1/\rho_1,\dots,T_n/\rho_n\}}
\newcommand{\Rc}{R^\circ}

\newtheorem{theorem}{Theorem}
\theoremstyle{definition}
\numberwithin{theorem}{section}
\newtheorem{lemma}[theorem]{Lemma}
\newtheorem{definition}[theorem]{Definition}
\newtheorem{prop}[theorem]{Proposition}

\newtheorem{remark}[theorem]{Remark}

\newtheorem{corollary}[theorem]{Corollary}
\newtheorem{hypothesis}[theorem]{Hypothesis}

\title{Properties of Extended Robba Rings}
\author{Peter Wear}

\begin{document}
\setlength{\unitlength}{1in}
\maketitle

\begin{abstract}
We extend the analogy between the extended Robba rings of $p$-adic Hodge theory and the one-dimensional affinoid algebras of rigid analytic geometry, proving some fundamental properties that are well known in the latter case. In particular, we show that these rings are regular and excellent. The extended Robba rings are of interest as they are used to build the Fargues-Fontaine curve.
\end{abstract}

\section{Introduction}

Since being introduced in \cite{FF}, the Fargues-Fontaine curve has quickly become an important object in number theory. Given a finite extension $K$ of $\Qp$ with Galois group $G_K$, an important result from $p$-adic Hodge theory gives an equivalence of categories between continuous representations of $G_K$ on finite free $\Zp$-modules and \'etale $(\phi,\Gamma)$-modules over the period ring $\mathbf{A}_K$ (see \cite{BC} for an exposition of this result). In \cite{FF}, Fargues and Fontaine describe this category in terms of vector bundles on the scheme-theoretic Fargues-Fontaine curve.

There is also an adic version of the Fargues-Fontaine curve, this is an analytification satisfying a version of the GAGA principle as seen in \cite[\S 4.7]{KL2}. Both versions of the curve parametrize the untilts of characteristic $p$ perfectoid fields. Recently, Fargues has formulated a conjecture using the curve to link $p$-adic Hodge theory, the geometric Langlands program and the local Langlands correspondence \cite{Far}.

The adic version of the curve is built out of extended Robba rings. These rings appear in $p$-adic Hodge theory (\cite{KL}, for example). In \cite{NP}, Kedlaya proved that they are strongly noetherian. Kiehl's theory of coherent sheaves on rigid analytic spaces has been extended to a similar theory on adic spaces by Kedlaya and Liu in \cite{KL2} and on rigid geometry by Fujiwara and Kato in \cite{FK}. The strong noetherian property is required to fit the curve into this theory.

This work suggests an analogy between the extended Robba rings and one-dimensional affinoid algebras. In \cite{NP}, Kedlaya established some finer properties of the rings suggested by this analogy and listed some other expected properties \cite[Remark 8.10]{NP}. In this paper, we establish these properties. We hope that the extension of this analogy will help transfer results from the theory of rigid analytic spaces to the Fargues-Fontaine curve. 

We now give an outline of this paper. The extended Robba rings are completions of rings of generalized Witt vectors. In \cite{NG}, Kedlaya gave a classification of the points of the Berkovich space associated to $W(R)$ - the ring of Witt vectors over any perfect $\Fp$-algebra $R$. His proof demonstrates a close analogy between $W(R)$ and the polynomial ring $R[T]$ equipped with the Gauss norm. We first extend this classification to the generalized Witt vectors by exploiting the many shared functorial properties of the two constructions. We then consider higher rank valuations to get a description of the corresponding adic space, again taking advantage of the analogy to $R[T]$. Completing, we get the classification for the extended Robba rings.

Using this explicit classification, we can prove that the rank-$1$ valuations of these rings are dense in the constructible topology of the adic spectrum. This allows us to compute the power bounded elements of a rational localization of these rings. We then extend these results to the rings defined by \'etale morphisms of extended Robba rings. 

Finally, we prove a form of the Nullstellensatz and use this to prove regularity and excellence for these rings. In characteristic zero, our proof of excellence is an adaption of the proof in Matsumura's book \cite[Theorem 101]{Mat1} that the ring of convergent power series over $\R$ or $\C$ is excellent. In particular, we work with the derivations of the rings, proving a Jacobian criterion. In the characteristic $p$ case, excellence follows from a theorem of Kunz \cite[Theorem 2.5]{Kun}.

\subsection{Acknowledgements}

The author would like to thank Kiran Kedlaya for suggesting these questions and for many helpful conversations. The author gratefully acknowledges the support of NSF grant DMS-1502651 and UCSD.

\section{Generalized Witt Vectors}\label{GWV}

Throughout this paper, we will be working in the same setup as \cite{NP}. Fix a prime $p$ and a power $q$ of $p$. Let $L$ be a perfect field containing $\Fq$, complete for a multiplicative nonarchimedean norm $|\bullet|$, and let $E$ be a complete discretely valued field with residue field containing $\Fq$ and uniformizer $\varpi\in E$. Let $\O_L$ and $\O_E$ be the corresponding valuation subrings and write $W(\O_L)_E:=W(\O_L)\otimes_{W(\Fq)} \O_E$. Concretely, each element of $\WLE$ can be uniquely written in the form $\sum_{i\geq 0} \varpi^i[\overline{x}_i]$ with $\overline{x}_i\in\O_L$. 

This ring is treated at length in \cite[Sections 5-6]{FF} with the notation $W_{\O_E}(\O_F)$. Alternately, $W(\O_L)_E$ is a ring of generalized Witt vectors as described in \cite[Section 2]{CD} with the notation $W_\varpi(\O_L)$. The generalized Witt vectors retain many useful properties of the usual $p$-typical Witt vectors. In this section, we briefly go over the results we will need in the rest of the paper.

Recall that given a perfect $\Fp$-algebra $R$, we can define $W(R)$ functorially as as the unique strict $p$-ring $W(R)$ for which $W(R)/(p)\cong R$. The analogous statement is true for $\WLE$.

\begin{lemma}\label{strictring}
We have $\WLE/(\varpi)=\O_L$, and $\WLE$ is $\varpi$-torsion-free, $\varpi$-adically complete and separated.
\end{lemma}

\begin{proof}
\cite[Proposition 2.12]{CD}
\end{proof}

\begin{lemma}\label{witthomog}
The addition law is given by $$\displaystyle\sum_{i\geq 0} [\overline{x}_i]\varpi^i +\displaystyle\sum_{i\geq 0} [\overline{y}_i]\varpi^i =\displaystyle\sum_{i\geq 0} [\overline{z}_i]\varpi^i$$ where $\overline{z}_i$ is a polynomial in $\overline{x}_j^{q^{j-i}},\overline{y}_j^{q^{j-i}}$ for $j=0,\dots,i$. This polynomial has integer coefficients and is homogeneous of degree $1$ for the weighting in which $\overline{x}_j,\overline{y}_j$ have degree $1$. The analogous statement is true for multiplication.
\end{lemma}

\begin{proof} 
\cite[Remarque 5.14]{FF}.
\end{proof}

We conclude this section with a result on factoring in $\WLE$.

\begin{definition}\label{stable}
An element $x\in\WLE$ is \emph{stable} if it has the form $\sum_{i=0}^\infty \varpi^i[\overline{x_i}]$ with either $|\overline{x_i}|=0$ for all $i\geq 0$ or $|\overline{x_0}|>p^{-i}|\overline{x_i}|$ for all $i>0$. 
\end{definition}

\begin{theorem}\label{Wfactor}
Assume that $L$ is algebraically closed. For $x\in\WLE$ nonzero and not stable, we can write $x=y(\varpi-[u_1])\cdots (\varpi-[u_n])$ for some nonzero stable $y\in \WLE$ and some $u_1,\dots,u_n\in\O_L$ with $|u_1|,\dots,|u_n|\leq p^{-1}$.
\end{theorem}

\begin{proof}
\cite[Th\'eor\`eme 6.46]{FF}.
\end{proof}

\section{The Berkovich spectrum of the generalized Witt vectors}

Let $R$ be a perfect $\Fp$-algebra, equipped with the trivial norm. In \cite[Theorem 8.17]{NG}, Kedlaya gives an explicit classification of the points of $\mathcal{M}(W(R))$. In this section, we will extend this result to $\mathcal{M}(W(\O_L)_E)$. This entire section follows Kedlaya's paper extremely closely, the arguments all carry over fairly directly due to the similarities of the rings $W(R)$ and $\WLE$ stated in Section \ref{GWV}. We therefore explain the differences caused by the change of rings, but when arguments are essentially identical to the original paper we simply give a sketch and a reference to the original proof.

We first define an analogue of the Gauss seminorm on $\WLE$.
\begin{lemma}\label{gaussnorm}
The function $\lambda:\WLE\ra\R$ given by $$\lambda \Big(\displaystyle\sum_{i=0}^\infty \varpi^i[\overline{x}_i]\Big)=\displaystyle\max_{i}\big\{p^{-i}|\overline{x}_i|\big\}$$ is multiplicative and bounded by the $p$-adic norm.
\end{lemma}

\begin{proof}
This is the analogue of \cite[Lemma 4.1]{NG} where the seminorm $\alpha$ (which is $|\cdot|$ in our case) is assumed to be multiplicative. The proof is identical to the normal Witt vector case, using the fact that addition and subtraction are defined on the $j$th Teichmuller component as homogeneous polynomials of degree $q^j$ as in Lemma \ref{witthomog}.
\end{proof}

This acts like the $(p^{-1})$-Gauss seminorm for the generator $\varpi$, and the $(r/p)$-Gauss seminorm can be defined by replacing $p^{-i}$ in the above lemma by $(r/p)^i$. We can use this to build analogues of Gauss seminorms for other generators and weights.

\begin{definition}\label{hdef}
Given $u\in \O_L$ with $|u|\leq p^{-1}$, let $\pi=\varpi-[u]$. Then given $r\in[0,1]$, we define the valuation $H(u,r)$ to be the quotient norm on $\WLE[T]/(T-\pi)\cong \WLE$ induced by the $(r/p)$-Gauss extension of $|\cdot|$ to $W(R)[T]$.
\end{definition}

To show that these valuations are multiplicative and to compute them easily, we define stable presentations.

\begin{definition}\label{stablepres}
As $L$ is complete, so is $\O_L$ and so $\WLE$ is $(\varpi, [u])$-adically complete, so any sum $\sum_{i=0}^\infty x_i\pi^i$ with $x_i\in\WLE$ converges to some limit $x$.  We say that the sequence $x_0,x_1,\dots$ forms a presentation of $x$ with respect to $u$. If each $x_i$ is stable (Definition \ref{stable}), we call this a \emph{stable presentation}.
\end{definition}

\begin{lemma}\label{stablemma}
For any $x\in\WLE$, there exists some $y=\sum^\infty_{i=0} \varpi^i[\overline{y_i}]\in\WLE$ with $x\equiv y \pmod \pi$ and $|\overline{y_0}|\geq |\overline{y_i}|$ for all $i>0$. By definition, $y$ is stable.
\end{lemma}

\begin{proof}
This follows the construction of \cite[Lemma 5.5]{NG} but is a bit simpler as our $\pi$ is of the form $\varpi-[u]$ instead of some general primitive element. For any integer $j\geq 0$, we can write $$x=\displaystyle\sum_{i=0}^\infty \varpi^i[\overline{x_i}]\equiv \displaystyle\sum_{i=0}^j [u]^i[\overline{x_i}]+\displaystyle\sum_{i=j+1}^\infty \varpi^i[\overline{x_i}]\pmod \pi.$$ 
As $j$ grows, $\sum_{i=j+1}^\infty \varpi^i[\overline{x_i}]$ goes to zero. So either there exists some $N>0$ such that $|\sum_{i=0}^N [u]^i[\overline{x_i}]|\geq |x_n|$ for all $n>N$ or the sum $\sum_{i=0}^\infty [u]^i[\overline{x_i}]$ converges. In the first case, we let $y=\displaystyle\sum_{i=0}^N [u]^i[\overline{x_i}]+\displaystyle\sum_{i=N+1}^\infty \varpi^i[\overline{x_i}],$ in the second case, we let $y=\displaystyle\sum_{i=0}^\infty [u]^i[\overline{x_i}].$
\end{proof}

\begin{lemma}\label{stabexist}
Every element of $\WLE$ admits a stable presentation.
\end{lemma}

\begin{proof}
This is the analogue of \cite[Lemma 5.7]{NG}. Given $x,x_0,x_1,\dots, x_{i-1}\in \WLE$, apply Lemma \ref{stablemma} to construct $x_i$ congruent to $(x-\sum_{j=0}^{i-1}x_j)/\pi^i\pmod \pi$. This process yields a stable presentation of $x_0,x_1,\dots$ of $x$.
\end{proof}

\begin{theorem}\label{hmult}
The function $H(u,r)$ is a multiplicative seminorm and bounded by $\lambda$. Given any stable presentation $x_0,x_1,\dots$ of $x\in \WLE$, $H(u,r)=\max_i\{(r/p)^i\lambda(x_i)\}$.
\end{theorem}

\begin{proof}
The proof of \cite[Theorem 5.11]{NG} carries over exactly as all the needed properties of presentations in $W(R)$ also hold in $\WLE$. 
\end{proof}

The following computation will be useful later.

\begin{corollary}\label{hcalc}
For $u,u'\in \O_L$ with $|u|,|u'|\leq p^{-1}$ and $r\in[0,1]$, $$H(u,r)(\varpi-[u'])=\text{max}\{ r/p, H(u,0)(\varpi-[u'])\}.$$
\end{corollary}

\begin{proof}
\cite[Lemma 5.13]{NG}
\end{proof}

\begin{remark}\label{vzero}
\cite[Remark 5.14]{NG} As $H(u,0)$ is the quotient norm on $\WLE/(\pi)$ induced by $\lambda$, we have $H(u,0)(x)=0$ if and only if $x$ is divisible by $\pi$. 

Furthermore, any $v\in\mathcal{M}(\WLE)$ with $v(\pi)=0$ must equal $H(u,0)$. Given $x\in\WLE$, we can construct a stable presentation with respect to $\pi$. Then only the first term of the presentation will affect $v(x)$ as $v(\pi)=0$, so $v$ is exactly $H(u,0)$ by Theorem \ref{hmult}.
\end{remark}

\begin{lemma}\label{vzerorestrict}
For any $v\in\mathcal{M}(\WLE)$, there exists a perfect overfield $L'$ of $L$ complete with respect to a multiplicative nonarchimedean norm extending the one on $L$ and some ${u}\in\mathfrak{m}_{L'}\setminus\{0\}$ such that the restriction of $H({u},0)$ to $\WLE$ equals $v$.

\end{lemma}

\begin{proof}
This is shown in \cite[Lemma 6.3]{NP} for the rings $\Bi$ defined in \ref{ABdef}, the exact same proof will work for $\WLE$. The analogous construction for $p$-typical Witt vectors is in \cite[Definition 7.5]{NG}.
\end{proof}

This lemma is very important as it allows us to reduce our study of general seminorms of $\WLE$ to those in Definition \ref{hdef}. 

\begin{definition}\label{def:retract}
Given $v\in\mathcal{M}(\WLE)$ and $\rho\in[0,1]$, choose $L',{u}$ as in \ref{vzerorestrict} and define $H(v,\rho)$ to be the restriction of $H({u},\rho)$ to $\WLE$. We define the \emph{radius} of $v$ to be the largest $\rho\in[0,1]$ for which $H(v,\rho)=v$. This is well defined by continuity.
\end{definition}

\begin{lemma}\label{lemma:retract}
This definition doesn't depend on $L'$ or ${u}$ and defines a continuous map $H:\mathcal{M}(\WLE)\times [0,1]\ra\mathcal{M}(\WLE)$ such that $$H(H(v,\rho),\sigma)=H(v,\max\{\rho,\sigma\})\quad (v\in\mathcal{M}(\WLE);\rho,\sigma\in[0,1]).$$
\end{lemma}

\begin{proof}
\cite[Theorem 7.8]{NG}
\end{proof}

Now let $\tilde{L}$ be a completed algebraic closure of $L$, there is a unique multiplicative extension of $|\cdot|$ to $\tilde{L}$ so we will continue to call this $|\cdot|$. Let $\O_{\tilde{L}}$ be the valuation ring of $\tilde{L}$ and equip $W(\O_{\tilde{L}})$ with the multiplicative norm $\tilde\lambda$.

\begin{definition}\label{vtilde}
For $u\in\O_{\tilde{L}}$ with $|u|\leq p^{-1}$ and $r\in [0,1]$, let $\tilde{\beta}_{u,r}$ be the valuation $H(u,r)$ and let $\beta_{u,r}$ be the restriction of $\tilde{\beta}_{u,r}$ to $\WLE$.
\end{definition}

\begin{remark}
There is a natural analogue of $\tilde\beta_{u,r}$ in $\mathcal{M}(K[T])$ where $K$ is an algebraically closed field. In that case, the seminorm can be identified with the supremum norm over the closed disc in $\C$ of center $u$ and radius $r$. An analogous statement holds here, Lemma \ref{balldef} implies that $\tilde\beta_{u,r}$ dominates the supremum norm. We won't use or prove this fact, but it may be helpful for intuition.
\end{remark}

We now give a very brief exposition of some useful properties of the valuations $\tilde{\beta}_{u,r}$ and $\beta_{u,r}$ that are needed for the classification. All proofs are now identical to those in \cite{NG}. By \cite[Lemma 8.3]{NG}, we have $\tilde{\beta}_{u,r}=\tilde{\beta}_{u',r}$ if and only if $r/p\geq\tilde{\beta}_{u',0}(\varpi-[u])$. We can therefore replace the center $u$ of $\tilde{\beta}_{u,r}$ with a nearby element $u'\in\O_{\tilde{L}}$ \cite[Corollary 8.4]{NG} which we can choose to be integral over $\O_L$ \cite[Corollary 8.5]{NG}. This integrality allows us to move to $\beta_{u,r}$: factoring the minimal polynomial of $u$ reduces computations to checking things of the form $p-[u_i]$. This type of argument implies that the radius works as expected, the radius of $\beta_{u,r}$ is $r$ \cite[Corollary 8.8]{NG}. 

This brings us to the key lemma for our classification.

\begin{lemma}\label{bigrad}
Given $v\in\mathcal{M}(\WLE)$ with radius $r$ and $s\in (r,1]$, there exists $u\in \O_{\tilde{L}}$ with $|u|\leq p^{-1}$ for which $H(v,s)=\beta_{u,s}$.
\end{lemma}

\begin{proof}
The full proof is given in \cite[Lemma 8.10]{NG}, we will give a sketch. The set of $s$ such that $H(v,s)=\beta_{u,s}$ for some $u\in\O_{\tilde{L}}$ is up-closed and nonempty. Let $t$ be its infimum, we will check that  $t\leq r$.

By Lemma \ref{vzerorestrict}, we can expand $L$ to some $L'$ and find $w\in\mathfrak{m}_{L'}\setminus\{0\}$ such that $v$ is the restriction of $H(w,0)$ to $\WLE$. Taking an algebraic closure $\overline{L'}$ lets us identify $\O_{\overline{L}}$ with a subring of $\O_{\overline{L'}}$, so we can compare $\beta_{w, s}$ with any $\beta_{u,s}\in\mathcal{M}(W(\O_{\tilde{L}}))$. If $s<t$ then these valuations must be distinct, so $s/p<\beta_{w,0}(p-[u])$.

By the computation in Corollary \ref{hcalc}, $$\beta_{w,s}(p-[u])=\text{max}\{s/p,\beta_{w,0}(p-[u])\}=\beta_{w,0}(p-[u])$$ when $s\leq t$, so for these elements the valuation doesn't depend on $s$. But because $\overline{L'}$ is algebraically closed, by Theorem \ref{Wfactor} we can factor every element of $\WLE$ into the product of a stable element and finitely many $p-[u_i]$ . As all of the terms of this product are independent of $s$, we conclude that $\beta_{w,s}=\beta_{w,0}=v$ for all $s\in [0,t]$ and so $t\leq r$ as desired.
\end{proof}

\begin{lemma}\label{balldef}
For $u\in\O_{\tilde{L}}$ with $|u|\leq p^{-1}$ and $r\in [0,1]$, let $D(u,r)$ be the set of $\beta_{v,0}$ dominated by $\beta_{u,r}$. Then for $r,s\in [0,1]$, $D(u,r)=D(u,s)$ if and only if $r=s$.
\end{lemma}

\begin{proof}
\cite[Lemma 8.16]{NG}
\end{proof}

\begin{theorem}\label{berkclass}
Each element of $\mathcal{M}(\WLE)$ is of exactly one of the following four types.

\begin{enumerate}
\item A point of the form $\beta_{u,0}$ for some $u\in\O_{\tilde{L}}$ with $|u|\leq p^{-1}$. Such a point has radius $0$.
\item A point of the form $\beta_{u,r}$ for some $u\in\O_{\tilde{L}}$ with $|u|\leq p^{-1}$ and some $r\in (0,1]$ such that $t/p$ is the norm of an element of $\O_{\tilde{L}}$. Such a point has radius r.
\item A point of the form $\beta_{u,r}$ for some $u\in\O_{\tilde{L}}$ with $|u|\leq p^{-1}$ and some $r\in (0,1)$ such that $t/p$ is the not norm of an element of $\O_{\tilde{L}}$. Such a point has radius r.
\item The infimum of a sequence $\beta_{u_i,r_i}$ for which the sequence $D(u_i,r_i)$ is decreasing with empty intersection. Such a point has radius $\displaystyle\inf_i{r_i}>0$.
\end{enumerate}

\end{theorem}

\begin{proof}
This is \cite[Lemma 8.17]{NG}, we again give a sketch. Types (i), (ii), (iii) are distinct as they have different radii, and $\beta_{u,r}$ cannot be type (iv) because $\beta_{u,0}$ would be in each $D(u_i,t_i)$. So the four types of points are distinct and we must check that any $\beta$ not of the form $\beta_{u,s}$ is type (iv).

Let $r$ be the radius of $\beta$, choose a sequence $1\geq t_1>t_2>\cdots$ with infimum $r$. Then by \ref{bigrad} we have $H(\beta,t_i)=\beta_{u_i,t_i}$ for some $u_i\in\O_{\tilde{L}}$. Then $\beta_{u_1,t_1},\beta_{u_2,t_2},\dots$ is decreasing with infimum $\beta$, so the sequence $D(u_i,t_i)$ is also decreasing. Any $u$ in the intersection would allow us to write $\beta=\beta_{u,r}$, so the $D(u_i,t_i)$ must have empty intersection. The radius must be nonzero because any decreasing sequence of balls with empty intersection must have radii bounded below by a nonzero number. 
\end{proof}

\section{The points of the adic spectrum}

We will now add in a classification of the higher rank valuations, giving a complete description of the adic spectrum.  To move to $\Spa(W(\O_L)_E,\WLE^\circ)$, we must also consider higher rank valuations. In the lecture notes \cite[Lecture 11]{Con}, Conrad gives an explicit description of the points of the adic unit disk $\Spa(k\langle t \rangle, k^\circ \langle t\rangle)$ over a non-archimedean field $k$. In this section, we adapt this proof to show that all the higher rank points are the natural equivalent of the type $5$ points of the adic unit disk. 

Define the abelian group $\Gamma:=\R_{>0}\times \Z$ with the lexicographical order, the group action given by $(t_1,m_1)(t_2,m_2)=(t_1t_2,m_1+m_2)$. We define $1^-:=(1,-1)$, $1^+:=(1,1)=1/1^-$, $r^-=r1^-$, and $r^+=r1^+$. Intuitively, $1^-$ is infinitesimally less than 1.

\begin{definition}\label{type5}
Given $u\in\O_{\tilde{L}}$ with $|u|\leq p^{-1}$ and $r\in (0,1]$, we define $\tilde{\beta}_{u,r^+}: W(\O_{\tilde{L}})_E\ra\Gamma\cup \{0\}$ by $$\tilde{\beta}_{u,r^+}(x)=\displaystyle\max_i\big\{(r^+/p)^i|x_i|\}$$ for any stable presentation $x_0,x_1,\dots$ of $x$ with respect to $u$. We define $\beta_{u,r^+}$ to be the restriction of $\tilde{\beta}_{u,r^+}$ to $\WLE$ and we define $\tilde{\beta}_{u,r^-}$ and $\beta_{u,r^-}$ analogously. We call these the type $5$ valuations.
\end{definition}

We remark that this definition doesn't depend on the choice of stable presentation by the argument of \cite[Theorem 5.11]{NG}, this was already used in Theorem \ref{hmult}. One can check that these valuations are continuous. If $r/p$ isn't the norm of an element of $W(\O_{\tilde{L}})_E$ then $\beta_{x,r^+}=\beta_{x,r^-}=\beta_{x,r}$. We will therefore assume from now on that $|r/p|\in |W(\O_{\tilde{L}})_E^\times|=|\O_L^\times|$.

\begin{remark}\label{5uniquemax}
Unlike the rank one case, the maximum in the definition of $\tilde{\beta}_{u,r^+}$ and $\tilde{\beta}_{u,r^-}$ is attained by a unique element. This is clear as any two terms have different powers of $1^+$.
\end{remark}

\begin{theorem}\label{only5}
All of the points of $\Spa(\WLE,\WLE^\circ)\setminus \mathcal{M}(\WLE)$ are of type 5.
\end{theorem}

\begin{proof}
The argument is essentially that of \cite[Theorem 11.3.13]{Con}, we give a sketch pointing out the differences. We assume that $L$ is algebraically closed; as before the general case follows by restricting valuations from the algebraic closure. Let $v$ be a valuation of $\WLE$ with rank greater than 1. Then there is some $x'\in \WLE$ with $v(x')\not\in |\O_L^\times|$. As $L$ is algebraically closed, we can use Theorem \ref{Wfactor} to factor $x'=y(p-[u_1])\cdots(p-[u_n])$ for $y\in \WLE$ stable and $u_i\in\O_L$ with $|u_i|\leq p^{-1}$. We have $v(y)\in |\O_L^\times|$, so we must have some $u_i$ with $v(p-[u_i])\not\in |\O_L^\times|$. For simplicity we will call this $u$ and define $\pi=p-[u],\ \gamma=v(p-[u_i])$. We will show that $u$ must act as our center.

As $L$ is algebraically closed, $|\O_L^\times|$ is divisible, so $\gamma^m\not\in |\O_L^\times|$ for any non-zero integer $m$. Then given any $x\in\WLE$, if we construct a stable presentation $x=\sum_{i=0}^\infty x_i\pi^i$ we have $$v(x)=\displaystyle\max_i v(x_i\pi^i)=\displaystyle\max_i \gamma^i|x_i|$$ as the valuations of the terms are pairwise distinct. The rest of the proof now follows exactly as in \cite{Con}. One checks that $\gamma$ must be infinitesimally close to some $r\in |\O_L^\times|$ (if this weren't the case we could construct an order-preserving homomorphism from $\Gamma\ra \R_{>0}$, contradicting that $\Gamma$ is higher rank). Then either $v=\beta_{u,r^+}$ or $\beta_{u,r^{-}}$ depending on if $\gamma>r$ or $\gamma<r$ in $\Gamma$.
\end{proof}

\section{The extended Robba rings}

We now define the rings used in the construction of the Fargues Fontaine curve and extend the classification of \ref{berkclass} and \ref{only5} to them. 

\begin{definition}\label{ABdef} Following \cite[Definition 2.2]{NP}, we define $$A_{L,E}=W(\O_L)_E[[\overline{x}]:\overline{x}\in L], \quad B_{L,E}=A_{L,E}\otimes_{\O_E} E.$$ 
\end{definition}

Each element of $A_{L,E}$ (resp. $B_{L,E}$) can be written uniquely in the form $\sum_{i\in\Z}\varpi^i[\overline{x}_i]$ for some $\overline{x}_i\in L$ which are zero for $i<0$ (resp. for $i$ sufficiently small) and bounded for $i$ large. The valuations $H(0,r)$ for $r\in(0,1]$ (defined in \ref{hdef}) therefore extend naturally to these rings, allowing us to make the following definition.

\begin{definition}\label{ABrdef}
Let $A^r_{L,E}$ be the completion of $A_{L,E}$ with respect to $H(0,r)$ and define $B^r_{L,E}$ analogously. Given a closed subinterval $I=[s,r]$ of $(0,\infty)$, let $\lambda_I=\displaystyle\max\{H(0,s),H(0,r)\}$, this is a power multiplicative norm as the $H(0,r)$ are multiplicative. Let $\Bi$ be the completion of $B_{L,E}$ with respect to $\lambda_I$.
\end{definition}

\begin{remark}\label{ABrrem}
In \cite{FF} the rings $B^r_{L,E}$ are called $B^b$ and the rings $B^I_{L,E}$ are called $B_I$. In \cite{NP} the Gauss norms used to obtain these rings are written in a different form but are equivalent.
\end{remark}

\begin{prop}\label{ABclass}
The points of adic spectra of all the rings defined in \ref{ABdef} and \ref{ABrdef} can be classified into types $1$-$5$ as in \ref{berkclass} and \ref{only5}.
\end{prop}

\begin{proof}
Given a valuation $v\in\Spa(A_{L,E},A^\circ_{L,E})$ (resp. $\Spa(B_{L,E}, B^\circ_{L,E})$) and an element $x=\sum_{i\in\Z}\varpi^i[\overline{x}_i]$ in $A_{L,E}$ or $B_{L,E}$, there exists some $\overline{y}\in \O_L$ and $k\geq 0$ such that $\varpi^k x[\overline{y^{-1}}]\in \WLE$. As $v$ is by definition multiplicative, we have $v(x)=v(\varpi^k x[\overline{y^{-1}}])v([\overline{y}])v(\varpi)^{-k}$, so $v(x)$ is uniquely determined by the restriction of $v$ to $\WLE$. Taking completions won't add points to the adic spectrum, so the desired result also follows for $A^r_{L,E},\ B^r_{L,E},$ and $\Bi$.
\end{proof}

\section{Rational Localizations}

We can now use this explicit classification to determine some properties of $\Spa(\Bi,\Bic)$ that can be checked on $\mathcal{M}(\Bi)$. We first set some notation. Let $(\Bi, \Bic)\ra (C,C^+)$ be a rational localization, so there exist elements $f_1,\dots,f_n,g\in B_{L,E}$ generating the unit ideal in $\Bi$ such that $$Spa(C,C^+)=\{ v\in\Spa(\Bi,\Bic):v(f_i)\leq v(g)\neq 0\quad (i=1,\dots,n)\}.$$ By \cite[Lemma 2.4.13a]{KL} we have $$C=\Bi\{T_1/\rho_1,\dots,T_n/\rho_n\}/(gT_1-f_1,\dots,gT_n-f_n)$$ (as $\Bi$ is strongly noetherian \cite[Theorem 4.10]{NP}, this ideal is already closed so we don't need to take the closure) and by definition $$C^+=\{x\in C: v(x)\leq 1\quad (v\in\Spa(C,C^+))\}.$$ 

\begin{remark}\label{LocGens}
By \cite[Remark 2.4.7]{KL}, we can choose the defining elements $f_1,\dots,f_n,g$ of our rational localization to be elements of $B_{L,E}$. This is convenient as it can be difficult to work with general elements of $\Bi$ - they aren't necessarily all of the form $\sum_{i\in\Z}\varpi^i[\overline{x}_i]$.
\end{remark}

We start with a useful computation, showing that inequalities coming from a type-$5$ valuation continue to be true near that valuation. 

\begin{lemma}\label{2v5}
Given elements $x$ and $y$ in $B_{L,E}$ and a type-$5$ valuation $\beta_{u,r^+}$ in $\Spa(\Bi,\Bip)$, if $\beta_{u,r^+}(x)\leq \beta_{u,r^+}(y)$ then there exists some real number $s>r$ such that for every $r'\in (r,s)$, $\beta_{u,r'}(x)\leq \beta_{u,r'}(y)$. If we instead choose $\beta_{u,r^-}$ such that $\beta_{u,r^-}(x)\leq \beta_{u,r^-}(y)$, then there exists some real number $s<r$ such that for every $r' \in (s,r)$, $\beta_{u,r'}(x)\leq \beta_{u,r'}(y)$.
\end{lemma}

\begin{proof}
We will just prove the first statement, the other statement follows identically. We can assume that $L$ is algebraically closed as the other case follows by restricting the valuations. Let $\pi=\varpi-[u]$ and fix stable presentations $x_0,x_1,\dots$ and $y_0,y_1,\dots$ of $x$ and $y$, so $\beta_{u,r^+}(x)=\max_{i}\{(r^+/p)^i \lambda_I(x_i)\}$ and $\beta_{u,r^+}(y)=\max_{i}\{(r^+/p)^i \lambda_I(y_i)\}$. Let $j$ be the unique (by Remark \ref{5uniquemax}) index such that $(r^+/p)^j\lambda_I(y_j)=\beta_{u,r^+}(y)$. 

For any term $x_i$, we have $(r^+/p)^i\lambda_I(x_i)\leq (r^+/p)^j\lambda_I(y_j)$ by assumption. If $i\leq j$ then for all $r'>r$ we have $(r'/p)^i\lambda_I(x_i)\leq (r'/p)^j\lambda_I(y_j)$ so any choice of $s$ will retain the desired inequality in this case. If $i>j$, we must have a strict inequality 
$(r/p)^i\lambda_I(x_i)< (r/p)^j\lambda_I(y_j)$ as if this were an equality moving from $r$ to $r^+$ would increase the left side more than the right side. For any fixed $i$, there is some interval $(r,s_i)$ where this inequality remains strict. It is therefore enough to show that we only need to consider finitely many terms. But there are only finitely many $i$ such that $(1/p)^i\lambda_I(x_i)>(r/p)^j\lambda_I(y_j)$ and these are the only terms of our presentation of $x$ that could ever pass the leading term of $y$. 
\end{proof}

\begin{prop}\label{contop}
The rank one valuations $\mathcal{M}(\Bi)$ are dense in $\Spa(\Bi,\Bic)$ in the constructible topology.
\end{prop}

\begin{proof}
It is known (eg \cite[Definition 2.4.8]{KL}) that $\mathcal{M}(\Bi)$ is dense in $\Spa(\Bi,\Bic)$ in the standard topology. We must show that any subset of $\Spa(\Bi,\Bic)$ that is locally closed in the standard topology has nonempty intersection with $\mathcal{M}(\Bi)$. As we have a basis of rational subsets, it is enough to show that if we have rational subsets $\Spa(C,C^+)$ and $\Spa(D,D^+)$ of $\Spa(\Bi,\Bic)$ and we have some semivaluation $v\in \Spa(C,C^+)\setminus\Spa(D,D^+)$, then there is some semivaluation $v'\in(\Spa(C,C^+)\setminus\Spa(D,D^+))\cap \mathcal{M}(\Bi)$. 

If $v\in\mathcal{M}(\Bi)$ there is nothing to check, so we may assume that $v$ has rank greater than 1. By the above classification, $v$ must be a type-$5$ point. We will assume that $v=\beta_{u,r^+}$, a similar argument will take care of the other option of $v=\beta_{u,r^-}$. Then if the corresponding type-$2$ point $\beta_{u,r}$ is in $\Spa(C,C^+)\setminus\Spa(D,D^+)$ we are done, so we can assume this is not the case. Then we will use \ref{2v5} to show that there is some $s>r$ such that for all $r'\in (r,s)$, $\beta_{u,r'}$ is in $\Spa(C,C^+)\setminus\Spa(D,D^+)$.

As $\beta_{u,r}\not\in\Spa(C,C^+)\setminus\Spa(D,D^+)$, we either have $\beta_{u,r}\in\Spa(D,D^+)$ or $\beta_{u,r}\not\in\Spa(C,C^+)$. If we are in the first case, then by definition there are finitely many elements (possibly not all the defining elements of the rational subset) $f_i,g\in B_{L,E}$ such that $\beta_{u,r}(f_i)\leq \beta_{u,r}(g)\neq 0$ and $\beta_{u,r^+}(f_i)>\beta_{u,r^+}(g)$. Then for each $i$, by Lemma \ref{2v5} we have some interval $(r,s_i)$ such that for all $r'\in (r,s_i)$, $\beta_{u,r'}(f_i)> \beta_{u,r'}(g)\neq 0$. As we are assuming that $\Bip=\Bic$, there is some interval $(r,s')$ such that $\beta_{u,r'}\in\Spa(\Bi,\Bic)$ for all $r'\in (r,s')$. The intersection of this finite set of intervals is nonempty, giving valuations of types $2$ and $3$ that are also in $\Spa(C,C^+)\setminus\Spa(D,D^+)$ as desired. The other case follows similarly.
\end{proof}

\begin{corollary}\label{berkcover}
A finite collection of rational subspaces of $\Spa(\Bi,\Bic)$ forms a covering if and only if the intersections with $\mathcal{M}(\Bi)$ do so.
\end{corollary}

\begin{corollary}\label{berkrat}
A rational subspace of $\Spa(\Bi,\Bic)$ is determined by its intersection with $\mathcal{M}(\Bi)$.
\end{corollary}

\begin{prop}\label{pbloc} If $\Bip$ equals $\Bic$, the ring of power-bounded elements of $\Bi$, then for any rational localization $(\Bi,\Bip)\ra (C,C^+)$, one also has $C^+=C^\circ$.
\end{prop}

\begin{proof}
By definition we have $C^+\subset C^\circ$, so we must show that $C^\circ\subset C^+$. By the description of $C^+$ at the start of the section, this can be done by showing that for any $x\in C^\circ$ and $v\in\Spa(C,C^+)$, $v(x)\leq 1$. We claim that it is enough to check this for $v\in\Spa(C,C^+)\cap\mathcal{M}(\Bi)$, i.e. when $v$ is a rank-one valuation. Assume the contrary: that $C^+\neq C^\circ$ but $\Spa(C,C^+)\cap\mathcal{M}(\Bi)=\Spa(C,C^\circ)\cap\mathcal{M}(\Bi)$. Then we can choose some $f\in C^\circ\setminus C^+$ and add the condition that $v(fg)\leq v(g)$ to the defining inequalities of $(\Bi,\Bip)\ra (C,C^+)$. This will give a new rational localization $(\Bi,\Bip)\ra (C',C'^+)$ such that $\Spa(C',C'^+)\subsetneq \Spa(C,C^+)$ as we are simply enforcing an extra nontrivial inequality. But by our assumption, this inequality was already satisfied by all the elements of $\Spa(C,C^+)\cap\mathcal{M}(\Bi)$, so we have $\Spa(C,C^+)\cap\mathcal{M}(\Bi)=\Spa(C',C'^+)\cap\mathcal{M}(\Bi)$. This contradicts Corollary \ref{berkrat}, so we must have $C^\circ=C^+$ as desired. We remark that $\Spa(C,C^\circ)$ need not a priori be a rational localization, this is why we had to construct $\Spa(C',C'^+)$.

Now fix some $v\in\Spa(C,C^+)\cap \mathcal{M}(\Bi)$. By \cite[Lemma 6.3]{NP}, $v$ is the restriction of a norm of the form $H(u,0)$ on some perfect overfield $L'$ of $L$ (compare with Lemma \ref{vzerorestrict}). The inclusion $L\ra L'$ gives an inclusion $\Bi\ra B^I_{L',E}$ such that $B^{I,\circ}_{L',E}\cap\Bi=\Bic$. Let $(C',C'^+)$ denote the base extension of $(C,C^+)$ along $(\Bi,\Bic)\ra(B^I_{L',E},B^{I,\circ}_{L',E}),$ so $$C'=B^I_{L',E}\{T_1/\rho_1,\dots,T_n/\rho_n\}/(gT_1-f_1,\dots,gT_n-f_n)$$ where $f_1,\dots,f_n,g$ are simply the corresponding elements of $\Bi$ viewed as elements of $B^I_{L',E}$. As $C'^\circ\cap C=C^\circ$ and $v'\in\Spa(C',C'^+)$ restricts to $v$ on $C$, checking that $v'(x)\leq 1$ for all $x\in C'^\circ$ will give our result.

By \cite[Remark 5.14]{NP}, the norm $v'$ on $B^I_{L',E}$ is just the quotient norm on $B^I_{L',E}/(\varpi-[u])=\mathcal{H}(v')$ (compare with \ref{vzero}). So extending $v'$ to $C'$, we get the quotient norm on $C'/(\varpi-[u])$. By \cite[Lemma 7.3]{NP}, the map $B^I_{L',E}/(\varpi-[u])\ra C'/(\varpi-[u])$ is an isomorphism, so we have reduced to looking at a multiplicative norm on a field. In this case, it is clear that power bounded elements have norm at most 1, so we are done.
\end{proof}

\section{\'Etale Morphisms}

\'Etale morphisms of adic spaces were first defined and studied by Huber in \cite{Hub}. In \cite[Section 8]{NP}, Kedlaya gives some results on \'etale morphisms of extended Robba rings. In this section, we recall the setup of Huber and some of Kedlaya's results. We then extend the results of the previous sections to \'etale morphisms.

\begin{hypothesis}\label{etalehyp}
Let $(\Bi, \Bip)\ra (C,C^+)$ be a morphism of adic Banach rings which is \'etale in the sense of Huber \cite[Definition 1.6.5]{Hub}. In particular, $C$ is a quotient of $\Bi\{T_1/\rho_1,\dots,T_n/\rho_n\}$ for some $n$, so it is strongly noetherian by \cite[Theorem 4.10]{NP}.
\end{hypothesis}

By definition, we get induced morphisms $\Spa(C,C^+)\ra\Spa(\Bi, \Bip)$ and $\mathcal{M}(C)\ra\mathcal{M}(\Bi)$.

\begin{lemma}\label{etalestructure} 
There exist finitely many rational localizations $\{(C,C^+)\ra (D_i, D_i^+)\}_i$ such that $\cup_i \Spa(D_i,D_i^+)=\Spa(C,C^+)$ and for each $i$, $(\Bi,\Bip)\ra (D_i, D_i^+)$ factors as a connected rational localization $(\Bi, \Bip)\ra (C_i, C_i^+)$ followed by a finite \'etale morphism $(C_i, C_i^+)\ra (D_i, D_i^+)$ with $D_i$ also connected.
\end{lemma}

\begin{proof}
\cite[Lemma 2.2.8]{Hub}
\end{proof}

The above construction commutes with base extension, so when we extend from $L$ to $L'$ as in Lemma \ref{vzerorestrict} the above result is retained. 

\begin{prop}\label{DDed}
The rings $C_i$ are all principal ideal domains and the rings $D_i$ are all Dedekind domains.
\end{prop}

\begin{proof}
See \cite[Theorem 7.11(c)]{NP} and \cite[Theorem 8.3(b)]{NP}.
\end{proof}

We now start working towards a classification of the points of $\Spa(C,C^+)$. We have a map $\Spa(C, C^+)\ra \Spa(\Bi, \Bip)$ and a good understanding of the points of $\Spa(\Bi, \Bip)$ from Proposition \ref{ABclass}, so given a point $v\in\Spa(\Bi, \Bip)$ we will look at the preimage $\{w_j\}_{j\in J}$. We will show that this is a finite set of valuations with the same rank and radius as $v$.

\begin{prop}\label{etalerank}
Given a valuation $w\in\Spa(C, C^+)$ mapping to $v\in \Spa(\Bi, \Bip)$, the rank of $w$ is the same as the rank of $v$.
\end{prop}

\begin{proof}
By Lemma \ref{etalestructure} we can assume that $\Bi\ra C$ is a finite map of rings, so $C$ is a finitely generated $\Bi$-module. Let the value group of $v$ be $H$ and the value group  of $w$ be $G$, then $[G: H]$ is finite so the groups must have the same rank.
\end{proof}

\begin{prop}\label{etalefinite1}
Given a valuation $\beta\in\Spa(\Bi, \Bip)$, the preimage $\{\gamma_j\}_{j\in J}$ is a finite set.
\end{prop}

\begin{proof}
By Lemma \ref{etalestructure}, we can choose a neighborhood of $\beta$ so that we are working with a finite \'etale morphism. This now follows from \cite[Lemma 1.5.2c]{Hub}.
\end{proof}

We can now extend the definition of radius to valuations in $\M(C)$. Retaining the notation of the proof of \ref{etalefinite1}, the radius of $\beta$ is the maximal $r\in [0,1]$ such that the restriction of $\beta_{u, r}$ to $\Bi$ is $\beta$. It is therefore natural to define the radius $r_j$ of $\gamma_j$ to be the maximal $r_j\in [0,1]$ such that some element of the preimage of $\beta_{u,r}$ restricts to $\gamma_j$ on $C$.

\begin{prop}\label{etaleradius}
Given $\beta\in\M(\Bi)\subset\Spa(\Bi, \Bip)$ with radius $r$ and $\gamma_j$ in the preimage of $\beta$ with radius $r_j$, we have $r=r_j$.
\end{prop}

\begin{proof}
If $s>r$, $\beta_{u, s}$ doesn't restrict to $\beta$ on $\Bi$ so the preimage of $\beta_{u, s}$ won't contain $\gamma_j$. We therefore have $r\geq r_j$. To show that $r\leq r_j$, we note that extending the radius is continuous and that the preimage of $\beta_{u, s}$ in $\Spa(C', C'^+)$ maps to a subset of the finite set $\{\gamma_j\}\subset \Spa(C,C^+)$ for all $s\in [0,r]$. So by continuity, every $\gamma_j$ in the preimage of $\beta$ is the restriction of some element of the preimage of $\beta_{u,r}$ as desired.
\end{proof}

We finally extend to higher rank valuations.

\begin{prop}\label{etalefinite2}
Let ${\beta}_{u, r^\pm}$ be a type $5$ valuation as in Definition \ref{type5}. Then the preimage of ${\beta}_{u, r^\pm}$ is in bijection with the preimage of $\beta_{u, r}$.
\end{prop}

\begin{proof}
This follows from continuity and the fact that the size of the fibers is locally constant.
\end{proof}

\begin{prop}\label{Cdense}
If $(C,C^+)$ is \'etale over $(\Bi,\Bic)$, the rank one valuations $\mathcal{M}(C)$ are dense in $\Spa(C, C^+)$ in the constructible topology.
\end{prop}

\begin{proof}
The map $\Spa(C, C^+)\ra \Spa(\Bi, \Bic)$ is locally the composition of open immersions and finite \'etale maps. \'Etale maps are smooth and therefore open by \cite[Proposition 1.7.8]{Hub}, and finite maps are closed by \cite[Lemma 1.4.5]{Hub}, so finite \'etale maps send locally closed subsets to locally closed subsets. The same is certainly true of open immersions, so any subset $U$ in $\Spa(C, C^+)$ that is locally closed under the standard topology is mapped to a locally closed subset $V$ in $\Spa(\Bi, \Bic)$. By Proposition \ref{contop}, there is some rank one valuation $v\in V$, and by Proposition \ref{etalerank} the preimage of $v$ is made up of rank one valuations. 
\end{proof}

\begin{corollary}\label{Ccover}
If $(C,C^+)$ is \'etale over $(\Bi,\Bic)$, a finite collection of rational subspaces of $\Spa(C,C^+)$ forms a covering if and only if the intersections with $\mathcal{M}(C)$ do so.
\end{corollary}

\begin{corollary}\label{Crat}
If $(C,C^+)$ is \'etale over $(\Bi,\Bic)$, a rational subspace of $\Spa(C,C^+)$ is determined by its intersection with $\mathcal{M}(C)$.
\end{corollary}

\begin{prop}\label{Cpbloc}
If $C^+$ equals $C^\circ$, then for any rational localization $(C, C^+)\ra (D, D^+)$, one also has $D^+=D^\circ$.
\end{prop}

\begin{proof}
The proof of \ref{pbloc} carries over. By \ref{Cdense}, we can again reduce to checking specific inequalities on rank one valuations. We can therefore use Lemma \ref{etalestructure} and Proposition \ref{DDed} to assume that $C$ is a Dedekind domain. By extending $L$ to some $L'$, we reduce to checking norm with nonempty kernel. As $C$ is a Dedekind domain, the kernel is a maximal ideal. We are therefore again dealing with power bounded elements in a field with a multiplicative norm, where the desired inequalities are clear.

\end{proof}

\section{Consequences of the strong noetherian property}

In \cite[Theorem 3.2]{NP}, Kedlaya gives a proof that the ring $\Ar$ is strongly noetherian, i.e. that $\Ar\{T_1/\rho_1,\dots,T_n/\rho_n\}$ is noetherian for any nonnegative integer $n$ and $\rho_1,\dots,\rho_n>0$. This is done very explicitly, using the theory of Gr\"{o}bner bases to construct generators for a given ideal. In this section, we adapt this proof to give a version of the Nullstellensatz for the rings $\ATn$. We then use the Nullstellensatz to prove that these rings are regular. In the next section, we will use regularity to prove that the rings $\ATn$ are excellent. We also show that $\Ar$ is also strictly noetherian. As in \cite{NP}, these arguments can be generalized to the rings $\Bi$ and rings $C$ coming from \'etale extensions of $\Bi$ as in Definition \ref{etalehyp}.

\begin{definition}\label{nulldef}
Given a ring $R$ and a subring $A$, we say that the pair $(R,A)$ satisfies the \emph{Nullstellensatz condition} if every maximal ideal of $R$ restricts to a maximal ideal of $A$.
\end{definition}

\begin{remark}\label{munshi}
We use this name because Munshi proved and then used this property for $(F[x_1,\dots,x_n],F[x_1])$ for $F$ a field to give a proof of Hilbert's Nullstellensatz, his proof is the subject of \cite{May}.
\end{remark}

\begin{theorem}\label{rnull}
Let $A$ be a nonarchimedean Banach ring with a multiplicative norm $|\cdot|$. Assume further that $A$ is a strongly noetherian Euclidean domain; in particular, this holds for $A=\Ar$. Let $n$ be a positive integer and $\rho=(\rho_1,\dots,\rho_n)$ an $n$-tuple of positive real numbers such that  the value group of $\ATn\setminus \{0\}$ has finite index over the value group of $A^\times$. Then $\ATn$ satisfies the Nullstellensatz condition with respect to $A$. 
\end{theorem}

\begin{proof}
We begin with two reductions. For each $T_i$, there is some positive integer $e_i$ such that $|T_i^{e_i}|\in |A^\times|$, so we can write $\ATn$ as an integral extension of the ring $\Ar\{(T_1/\rho_1)^{e_1},\dots,(T_n/\rho_n)^{e_n}\}$. This ring has the same value group as $A^\times$, and by going up maximal ideals of $\ATn$ restrict to maximal ideals of this ring. We can therefore assume that the value groups of $\ATn$ and $A^\times$ are equal.

We also note that it is enough to show that $\imax\cap A\neq 0$ whenever $A$ is not a field. Given $x_1\in\imax\cap A$ with $x_1\neq 0$, the ring $\ATn/(x_1)=A/(x_1)\{T_1/\rho_1,\dots,T_n/\rho_n\}$ is strongly noetherian and $\imax/(x_1)$ is a maximal ideal. We can therefore find $x_2\in\imax/(x_1)\cap A/(x_1)$ and iterate this process until $A/(x_1,\dots,x_n)$ is a field. This must eventually happen as $A$ is noetherian, then $\imax\cap A=(x_1,\dots,x_n)$ is a maximal ideal of $A$ as desired. We remark that it isn't clear that $A/(x_1,\dots,x_n)$ must be a nonarchimedean field.

To show that $\imax\cap A\neq 0$, we will be combining the proofs of \cite[Theorem 3.2]{NP} (via \cite[Lemma 3.8]{NP}) and Lemma \cite[3.8]{KL0}. The first proof deals specifically with the rings $\ATn$ while the second proves this version of the Nullstellensatz for a similar ring.

The proof will use ideas from the theory of Gr\"obner bases and an idea of Munshi \cite{May}. We therefore begin by setting up the combinatorial construction. 
\begin{hypothesis}
Let $I=(i_1,\dots,i_n)$ and $J=(j_1,\dots,j_n)$ denote elements of the additive monoid $\Z^n_{\geq 0}$ of $n$-tuples of nonnegative integers. 
\end{hypothesis}

\begin{definition}\label{Znorder}
We equip $\Z^n_{\geq 0}$ with the componentwise partial order $\leq$ where $I\leq J$ if and only if $i_k\leq j_k$ for $i=1,\dots,n$. This is a \emph{well-quasi-ordering}: any infinite sequence contains an infinite nondecreasing sequence.

We also equip $\Z^n_{\geq 0}$ with the \emph{graded lexicographic} total order $\preceq$ for which $I\prec J$ if either $i_1+\cdots+i_n<j_1+\cdots j_n$, or $i_1+\cdots+i_n=j_1+\dots +j_n$ and there exists $k\in\{1,\dots,n\}$ such that $i_\ell=j_\ell$ for $l<k$ and $i_k<j_k$. Since $\preceq$ is a refinement of $\leq$, it is a well-ordering.
\end{definition}

The key properties for the proof is that $\preceq$ is a well-ordering refining $\leq$ and that for any $I$, there are only finitely many $J$ with $J\preceq I$. 

\begin{definition}\label{leadingterm}
For $x=\sum_I x_IT^I\in \ATn$, define the \emph{leading index} of $x$ to be the index of $I$ which is maximal under $\preceq$ for the property that $|x_IT^I|=|x|$, and define the \emph{leading coefficient} of $x$ to be the corresponding coefficient $x_I$.
\end{definition}

We proceed by contradiction: suppose that $\imax$ is a maximal ideal of $\ATn$ with $\imax\cap A=0$. As we assumed that $A$ is strongly noetherian, $\ATn$ is noetherian so $\imax$ is closed by \cite[Proposition 3.7.2/2]{BGR}. 

Define the projection map $\psi$ forgetting the constant term of $x\in\ATn$, it is a bounded surjective morphism of Banach spaces with kernel $A$. Then $\imax+A$ is a closed subspace of $\ATn$, and $V=\psi(\imax+A)$ is closed by the open mapping theorem \cite[$\mathsection$ 2.8.1]{BGR}. So $\psi$ induces a bounded bijective map of Banach spaces $\imax\ra V$; by the open mapping theorem $\psi^{-1}$ is also bounded. Define the \emph{nonconstant degree} $\deg'(x)=\deg(\psi(x))$ to be the leading index of $\psi(x)$. Define the \emph{leading nonconstant coefficient} of $x$ to be $x_{\deg'(x)}$.
 
We now follow the proof of \cite[3.2]{NP} but using $\deg'$ instead of $\deg$ in \cite[3.7]{NP} and beyond and using $|\psi(\cdot)|$ instead of $|\cdot|_\rho$. We obtain a finite set of generators $m_I=\sum m_{I,J}T^J$ for $\imax$ such that the leading index of $\psi(m_I)$ is $I$. The key fact here is that for any $x\in\imax$ with leading index $J$, there is some $m_I$ with $I\preceq J$.

Scale the $m_I$ by elements of $A$ so that they all have norm 1. Let $a_I=m_{I,I}$ be the leading coefficient of $m_I$, so now $|a_I|=1$. Define $\epsilon<1$ as in \cite[3.8]{NP} to be the largest possible norm of some coefficient $m_{I,J}$ with $I\prec J$. As the norm on $A$ is multiplicative, the ring $\O_A/I_A$ has no nonzero nilpotents so the nilradical is \{0\}. We can therefore choose a nonzero prime ideal $\ip$ of $\O_A/I_A$ not containing $\prod_{I\in S} a_I$. Choose any $\varpi\in \O_A$ reducing to a nonzero element of $\ip$, so $|\varpi|=1$. As $\imax\cap A=0$, we have $\varpi\not\in \imax$, so by maximality we can find $x_0\in A\{T_1/\rho_1,\dots,T_n/\rho_n\}$ such that $1+\varpi x_0\in \imax$.

\begin{lemma}\label{approx}
Let $S$ be the multiplicative system generated by the $a_I$. Given any $c \in S$ and $x \in \ATn$ with $c+\varpi x \in \imax$, there exists some $c'\in S$, $x'\in \ATn$ with $c'+\varpi x'\in \imax$ and $|\psi(x')|\leq\epsilon |\psi(x)|$.
\end{lemma}

\begin{proof}
We will construct $c'$ and $x'$ iteratively. Given any $c_\ell\in S$, $x_\ell\in \ATn$ with $c_\ell+\varpi x_\ell\in \imax$, we will construct $c_{\ell+1}\in S$, $x_{\ell+1}\in \ATn$ with $c_{\ell+1}+\varpi x_{\ell+1}\in \imax$ and $|\psi(x_{\ell+1})|\leq |\psi(x_\ell)|$. We will then use this construction to get $c'$ and $x'$.

Choose $\lambda\in A^\times$ so that $|\psi(\lambda (c_\ell+\varpi x_\ell))|=1$, and let $e_{I_\ell} T^{I_\ell}$ be the leading term of $\psi(\lambda x_\ell)$. Note that the leading index $I_\ell$ is the same as the leading index of $\psi(\lambda (c_\ell+\varpi x_\ell))$ as $\psi$ causes the contribution of $c_\ell$ to be forgotten and $\varpi$ is a nonzero element of $A$ so it won't affect the leading index. Then by the construction of the $m_I$ there is some $m_\ell$ with leading index $J_\ell\leq I_\ell$; let $a_\ell$ be the leading coefficient of $m_\ell$. Define $$y_\ell=a_{\ell}\lambda(c_\ell+\varpi x_\ell)-\varpi e_{I_\ell} m_\ell T^{I_\ell-J_\ell}\in \imax.$$ This has been chosen so that the coefficient of $T^{I_\ell}$ in $y_\ell$ is $0$ and $|\psi(y_\ell)|\leq 1$. Let $$x_{\ell+1}=\frac{\lambda^{-1} y_\ell-a_{\ell} c_\ell}{\varpi}=a_\ell x_\ell-\lambda^{-1}e_{I_\ell}m_\ell T^{I_\ell-J_\ell},\quad c_{\ell+1}=a_{\ell}c_{\ell}.$$ Clearly $c_{\ell+1}\in S$, $c_{\ell+1}+\varpi x_{\ell+1}=\lambda^{-1}y_{\ell+1}\in \imax$, and $|\psi(x_{\ell+1})|\leq |\psi(x_\ell)|$.

As $\preceq$ is a well ordering, we have a bijection between indices of $x$ and positive integers - call the $m$th index $I_m$. As $\psi(x_\ell)$ is a convergent power series, there are only finitely many terms of $\psi(x_\ell)$ with coefficient norm greater than $\epsilon|\psi(x)|$. We can therefore associate a unique integer $n_\ell$ to each $\psi(x_\ell)$ such that the $m$th term in the binary representation of $n_\ell$ is $1$ exactly when the coefficient of $T^{I_m}$ in $\psi(x_\ell)$ is greater than $\epsilon|\psi(x)|$. We claim that $n_\ell> n_{\ell+1}$ whenever $n_{\ell}>0$, so after finitely many steps we must have $n_k=0$. By definition, this means that after finitely many steps every term of $\psi(x_k)$ will have coefficient with norm at most $\epsilon|\psi(x)|$, so $|\psi(x_k)|\leq \epsilon|\psi(x)|$ as desired.

By the construction of $\epsilon$, adding the multiple of $m_\ell$ required to go from $x_\ell$ to $x_{\ell+1}$ won't introduce any coefficients with norm greater than $\epsilon|\psi(x)|$ and index $J\succ I_{\ell}$. By the construction of $x_{\ell+1}$, the coefficient of $T^{I_\ell}$ in $x_{\ell+1}$ is $0$. So when we move from $n_\ell$ to $n_{\ell+1}$, the digit corresponding to $I_\ell$ is changed from $1$ to $0$ and no higher digits are changed. So $n_\ell>n_{\ell+1}$ as desired.
\end{proof}

Starting with $c_0=1$, we can iterate this process to get sequences $\{c_\ell\} \subset S$, $\{x_\ell\}\subset  A\{ T_1,\dots,T_n\}$ so that for all $\ell$, $y_\ell:=c_\ell+\varpi x_\ell\in \imax$ and $|\psi(y_\ell)|\ra 0$. As the inverse of $\psi$ is bounded, we must have $|y_\ell|\ra 0$. This implies that $|c_\ell+\varpi x_{\ell,0}|\ra 0$ as this is the constant term of $y_\ell$, which implies that for $\ell$ large $c_\ell-\varpi x_{\ell,0}\in I_A$. This is a contradiction as we chose $\varpi$ so that $c_\ell$ is never divisible by $\varpi$ in $\O_A/I_A$. 
\end{proof}

\begin{remark}\label{bigvaluegroup}
We note that in our proof, it was essential that we could scale elements of $\ATn$ by elements of $A^\times$ to get elements of norm 1. The result isn't generally true if we allow infinite extensions in the value group. For example, if we let $A=\Qp\{T_1/\rho\}$ where $\rho$ is irrational, the ring $A\{T_2/\rho^{-1}\}$ has $(T_1T_2-1)$ as a maximal ideal, but $A\cap (T_1T_2-1)=\{0\}$. This is analogous to the more standard example of the maximal ideal $(px-1)$ of $\Zp[x]$.
\end{remark}

\begin{remark}\label{nullvals}
This result can be extended to some rings of the form $\Bi\{T_1/\rho'_1,\dots,T_n/\rho'_m\}$ by using Lemma \cite[4.9]{NP} to rewrite them as quotients of some $\ATrn$. As mentioned in Remark \ref{bigvaluegroup}, this argument will not hold for all of the $\Bi$, but it will work for intervals $I=[s,r]$ with $s\in\Q$. If the result does hold for $\Bi$, then it will also hold for any \'etale extension $C$. 
\end{remark}

\begin{corollary}\label{Breg}
The rings $\Ar\{T_1,\dots,T_n\},$ $\Bin$ are regular for $I$ as in Remark \ref{nullvals}.
\end{corollary}

\begin{proof}
We just show this for $\Bin$, the other case follows similarly. We must show that for any maximal ideal $\imax\subset\Bin$, the localization at $\imax$ is a regular local ring. By Theorem \ref{rnull}, $\imax\cap\Bi=(m)$ for some maximal ideal $(m)$ of the principal ideal domain $\Bi$. We claim that $\Bi/(m)$ is a nonarchimedean field; we must check that $\M(\Bi/(m))$ is a single point. By \cite[Lemma 7.10]{NP}, $\M(\Bi/(m))$ is a finite discrete topological space. By \cite[Proposition 2.6.4]{KL}, any disconnect of $\M(\Bi/(m))$ would induce a disconnect of $\Bi/(m)$. As $\Bi/(m)$ is a field, this is impossible so $\M(\Bi/(m))$ must be a point as desired. The ring $(\Bi/(m))\{T_1,\dots,T_n\}$ is therefore a classical affinoid algebra so it is regular \cite{Kie}. The result now follows from a general commutative algebra statement: given a local ring $(R,\imax)$ and an element $m\in\imax\setminus \imax^2$, if $R/(m)$ is regular then so is $R$. This is clear as in passing from $R$ to $R/(m)$, the dimension of the local ring is reduced by 1 by Krull's principal ideal theorem and the $k$-dimension of $\imax/\imax^2$ is reduced by 1 by our assumption on $m$.
\end{proof}

Using similar ideas, we show that $\Ar$ is strictly noetherian. We first recall the necessary definitions.

\begin{definition}\label{afg}
Let $(A,A^+)$ be a Huber pair with $A$ Tate. An $A^+$-module $N$ is \emph{almost finitely generated} if for every topologically nilpotent unit $u$ in $A$, there is a finitely generated $A^+$-submodule $N'$ of $N$ such that $uN$ is contained in $N'$.
\end{definition}

\begin{definition}\label{strictN}
A Huber pair $(A, A^+)$ is \emph{strictly noetherian} if for every finite $A^+$-module $M$, every $A^+$-submodule $N$ of $M$ is almost finitely generated. 
\end{definition}

\begin{remark}\label{strictNoethisNoeth}
We note that if $(A, A^+)$ is strictly noetherian, $A$ must be noetherian. Given an ideal $H$ of $A$ and topologically nilpotent unit $u$, $u(H\cap A^+)$ is contained in some finitely generated ideal $\langle x_1,\dots,x_n\rangle$ of $A^+$. For any $h\in H$, multiplying by a sufficiently large power of $u$ will give $u^nh\in u(H\cap A^+)$, so we have $u^nh=\sum a_ix_i$ with the $a_i$ in $A^+$, and so $h=\sum u^{-n}a_ix_i$ and so the $x_i$ generate $H$.
\end{remark}

\begin{remark}\label{KiehlRmk}
Kiehl was the first to consider the strict noetherian property, showing that affinoid algebras are strictly noetherian in \cite[Satz 5.1]{Kie2}.
\end{remark}

\begin{prop}\label{AstrictN}
For any nonnegative integer $n$ and $\rho_1,\dots,\rho_n\in\R_{>0}$, the pair $(R,R^\circ):=(\ATrn, \ATrn^\circ)$ is strictly noetherian.
\end{prop}

\begin{proof}
Quotients and direct sums preserve the almost finitely generated property, so it is enough to check that every ideal $H$ of $R^\circ$ is almost finitely generated. Fix any ideal $H\subset \Rc$ and topologically nilpotent unit $u$, we will construct a finitely generated ideal $H'\subset H$ such that $uH\subset H'$. Exactly following the construction of \cite[Theorem 3.2]{NP} gives a finite subset $\{x_I\}$ of $H$ such that for all $y\in H$, there exist $a_I\in R$ such that $|a_I|_\rho |x_I|_\rho\leq |y|_\rho$ for all $I$ and $y=\sum a_Ix_I$. Letting $\delta:=\displaystyle\min\{|x_I|_\rho\}$, we see that if $|y|_\rho\leq \delta$ we have $a_I\in R^\circ$ for all $I$. The set $\{x_I\}$ therefore generates all of the elements of $uH$ with norm at most $\delta$.

Let $c=|u|_\rho$, let $m=\lceil \log_c\delta\rceil$. As $u$ is topologically nilpotent we have $c<1$. For each index $I$ and $k\in \{1,\dots,m\}$, let $d_{I,k}$ be the smallest possible degree of the leading coefficient of an element of $H$ with leading index $I$ and weighted Gauss norm $c^k$. Following \cite[Definition 3.7]{NP}, for each nonnegative integer $d$ and $k\in\{1,\dots,n\}$, define $S_{d,k}$ to be the (finite) set of $I$ which are minimal with respect to $\leq$ for the property that $d_{I,k}=d$ and let $S_k$ be the union of the $S_{d,k}$. For each $I\in S_k$, choose $x_{I,k}\in H\setminus \{0\}$ with leading index $I$, weighted Gauss norm $c^k$, and leading coefficient $c_{I,k}$ of degree $d_I$. 

We claim that the finite set $\{x_I\}\cup \{x_{I,1}: I\in S_1\}\cup\dots\cup \{x_{I,m}: I\in S_m\}$ generates every element $y\in uH$. In the original proof, the key property of the chosen generators is that for any $y=\sum y_JT^J$ in the ideal with leading term $y_{J'}T^{J'}$, there is some $x_I$ such that $I\leq J'$ and $\deg(c_I)\leq \deg(y_{J'})$. This allows for a series of approximations that can be shown to converge using the fact that $\Ar$ is a Euclidean domain. 

This argument continues to hold in our case, but we must also use generators with norm at least that of $y$ so that at each step of the approximation we are multiplying $x_I$ by an element of $R^\circ$. Let $j=\lceil \log_c(|y|_\rho)\rceil$ and let $y_{J'}T^{J'}$ be the leading term of $y$. If $j>m$, $|y|_\rho\leq \delta$ so $y$ can be generated by elements of $\{x_I\}$. Otherwise, we have $u^{-1}y\in H$ and $|u^{-1}y|_\rho>c^k$, so there is some unit $v\in \Ar\cap \Rc$ with $|vu^{-1}y|_\rho=c^k$. Multiplication in $R$ by units will not change leading indices or leading degrees, so by construction, we can find some $x_{I,j}$ with $I\leq J'$ and $\deg(c_{I,j})\leq \deg(y_{J'})$. As $|x_{I,j}|_\rho=c^k\geq |y|_\rho$, this is the desired element, the rest of the proof is identical to \cite{NP}.
\end{proof}

\begin{corollary}\label{BCstrictN}
The rings $\Bi, C$ are strictly noetherian.
\end{corollary}

\begin{proof}
By \cite[Lemma 4.9]{NP} and Hypothesis \ref{etalehyp}, these rings are quotients of some ring of the form $\ATrn$.
\end{proof}

\section{Excellence}

Finally, we show that the rings $\ATrn, \Bi,$ and $C$ are excellent when the $\rho_i$ are chosen as in Theorem \ref{rnull}, adapting the argument of \cite[Theorem 101]{Mat1}. The idea is that the $n$ partial derivatives $\frac{\partial}{\partial T_i}$ of $\ATrn$ give us derivations which satisfy a Jacobian criterion that implies excellence. 

\begin{remark}\label{Echar0}
The methods of this section only work for rings of characteristic 0, but if $E$ is characteristic $p$ our rings will also be characteristic $p$. In this case, $\ATrn$ is a ring of convergent power series over $L$. As $L$ is perfect of characteristic $p$, $\ATrn$ has a finite $p$-basis given by $\varpi, T_1,\dots,T_n$. It is therefore excellent by a theorem of Kunz \cite[Theorem 2.5]{Kun}. Excellence of the other rings of interest follow as in the characteristic 0 case, see Corollary \ref{Bexc}.
\end{remark}

We begin the mixed characteristic case by recalling some results from Matsumura \cite[Sections 32, 40]{Mat1}.

\begin{definition}\label{J01def}
A noetherian ring $R$ is J-0 if $\Reg(\Spec(R))$ contains a non-empty open subset of $\Spec(R)$, and J-1 if $\Reg(\Spec(R))$ is open in $\Spec(R)$.
\end{definition}

\begin{lemma}\label{J2def}
For a noetherian ring $R$, the following conditions are equivalent: 
\begin{enumerate}
\item Any finitely generated $R$-algebra $S$ is J-1;
\item Any finite $R$-algebra $S$ is J-1;
\item For any $\ip\in \Spec(R)$, and for any finite radical extension $K'$ of the residue field $\kappa(\ip)$, there exists a finite $R$-algebra $R'$ satisfying $R/\ip\subseteq R'\subseteq K'$ which is J-0 and whose quotient field is $K'$.
\end{enumerate}

If these conditions are satisfied, we say that $R$ is J-2.
\end{lemma}

\begin{proof}
\cite[Theorem 73]{Mat1}
\end{proof}

\begin{corollary}\label{J2quot}
Given a noetherian ring $R$ containing $\Q$, if $R/\ip$ is J-1 for all $\ip\in \Spec(R)$ then $R$ is J-2.
\end{corollary}

\begin{proof}
By \cite[Chapter 32, Lemma 1]{Mat1}, the first condition in Lemma \ref{J2def} is equivalent to the following condition: Let $S$ be a domain which is finitely generated over $R/\ip$ for some $\ip\in\Spec(R)$, then $S$ is J-0. Let $k$ and $k'$ be the quotient fields of $R$ and $S$ respectively. As $\Q\subset R$, $k'$ is a separable extension of $k$ by \cite[Paragraph 23.E]{Mat1}. This is now Case 1 of the proof of \cite[Theorem 73]{Mat1}, so the result follows. 
\end{proof}

Our goal is to understand when quotients of a regular local ring are again regular. The following lemmas explain how to use derivations to do this. We first set some notation, following \cite[Section 40]{Mat1}.

Given a ring $A$, elements $x_1,\dots,x_r\in A$, and derivations $D_1,\dots,D_s\in\Der(A)$, we write $J(x_1,\dots,x_r;D_1,\dots,D_s)$ for the Jacobian matrix $(D_ix_j)$. Given a prime ideal $\ip\subset A$, we write $J(x_1,\dots,x_r;D_1,\dots,D_s)(\ip)$ for the reduction of the Jacobian mod $\ip$. If $\ip$ contains $x_1,\dots,x_r$, the rank of the Jacobian mod $\ip$ depends only on the ideal $I$ generated by the $x_i$, so we denote it $\rank J(I;D_1,\dots,D_s)(\ip)$. Given a set $\Delta$ of derivations of $A$, we define $\rank J(I;\Delta)(\ip)$ to be the supremum of rank $J(I;D_1,\dots,D_s)(\ip)$ over all finite subsets $\{D_1\dots,D_s\}\subset \Delta$.

\begin{lemma}\label{RegQuot}
Let $(R,\imax)$ be a regular local ring, let $\ip$ be a prime ideal of height $r$ and $\Delta$ be a subset of $\Der(R)$. Then:
\begin{enumerate}
\item $\rank J(\ip;\Delta)(\imax)\leq \rank J(\ip;\Delta)(\ip)\leq r$,
\item if $\rank J(f_1,\dots,f_r; D_1,\dots, D_r)(\imax)=r$ and $f_1,\dots,f_r\in\ip$, then $\ip=(f_1,\dots,f_r)$ and $R/\ip$ is regular.
\end{enumerate}

\end{lemma}

\begin{proof}
\cite[Theorem 94]{Mat1}
\end{proof}

\begin{lemma}\label{RegPrimeQuot}
Let $R$, $\ip$, and $\Delta$ be as in the preceding lemma. Then the following two conditions are equivalent:
\begin{enumerate}
\item $\rank J(\ip;\Delta)(\ip)=\hgt\ip$,
\item let $\iq$ be a prime ideal contained in $\ip$, then $R_\ip/\iq R_\ip$ is regular if and only if $\rank J(\iq;\Delta)(\ip)=\hgt\iq$.
\end{enumerate}
\end{lemma}

\begin{proof}
\cite[Theorem 95]{Mat1}
\end{proof}

\begin{definition}\label{WJdef}
The \emph{weak Jacobian condition} holds in a regular ring $R$ if for every $\ip$ in $\Spec(R)$, $\text{rank } J(\ip; \Der(R))(\ip)=\hgt \ip$. In this case we say that (WJ) holds in $R$.
\end{definition}

Mizutani and Nomura showed that rings satisfying (WJ) and containing $\Q$ are excellent: their proof is \cite[Theorem 101]{Mat1}. A key step in the proof is the following proposition, we go through it in detail as we will adapt it when showing that $A\{T_1,\dots, T_n\}$ is excellent. 

\begin{prop}\label{WJisJ2}
Every regular ring $R$ satisfying (WJ) is J-2.
\end{prop}

\begin{proof}
This is roughly the argument of \cite[Paragraph 40.D]{Mat1}. By Corollary \ref{J2quot}, it is enough to show that for every $\iq\in\Spec(R)$, the set $\Reg(R/\iq)$ is open in $\Spec(R/\iq)$. Fix any prime $\ip\supseteq \iq$ such that the image $P$ of $\ip$ in $\Spec(R/\iq)$ is regular. We will construct an open set around $P$ contained in $\Reg(R/\iq)$. 

As (WJ) holds in $R$, we have $\rank J(\ip;\Der(R))(\ip)=\rank J(\ip;\Der(R_\ip))=\hgt(\ip)$. By Lemma \ref{RegPrimeQuot}, this implies that $\rank J(\iq; \Der(R_\ip))(\ip)=\hgt(\iq)$ as $R_\ip/\iq R_\ip$ is regular. Let $r=\hgt(\iq)$, then we have $f_1,\dots,f_r\in\iq$ and $D_1,\dots,D_r\in \Der(R)$ such that $\det(D_if_j)\not\in(\ip)$. By Lemma \ref{RegQuot}, this implies that $\iq R_\ip=(f_1,\dots, f_r) R_\ip$. We therefore have some $g\in R-\ip$ such that $\iq R_g=(f_1,\dots,f_r) R_g$. Let $h=\det(D_if_j)$. By definition $g$ and $h$ are not in $\ip$ or $\iq$, so the reduction $\overline{gh}$ is nonzero in $R/\iq$. For any prime $\ip'$ reducing to some $P'\in D(\overline{gh})\subset\Spec(R/\iq)$, we have $\rank\ J(f_1,\dots, f_r; D_1,\dots, D_r)(\ip')=r$ as $\overline{h}\not\in P'$, so by Lemma \ref{RegQuot} $\hgt(\iq R_{\ip'})\geq r$. As $\overline{g}\not\in P'$, $f_1,\dots, f_r$ generate $\iq R_{\ip'}$ so $\hgt(\iq R_{\ip'})\leq r$. So $\hgt(\iq R_{\ip'})=r$ and we can apply Lemma \ref{RegQuot} to see that $R_{\ip'}/\iq R_{\ip'}$ is regular. So $\Reg(R/\iq R)$ contains the open set $D(\overline{gh})$ containing $P$, so it is open in $\Spec(R/\iq R)$ as desired.
\end{proof}

We can now state the hypotheses we need to prove excellence.

\begin{hypothesis}\label{exhyp}
Let $A$ be a regular integral domain containing $\Q$, let $R$ be a ring such that $A[T_1,\dots,T_n]\subset R\subset A[[ T_1,\dots,T_n]]$ and that is stable under the $n$ derivatives $\frac{\partial}{\partial T_i}$. Assume that $(R,A)$ satisfies the Nullstellensatz condition of Definition \ref{nulldef} and that $R\otimes_A \Fra(A)$ is weakly Jacobian as in Definition \ref{WJdef}.
\end{hypothesis}

It is clear that for (WJ) to hold in $R$ we must have $\dim_R(\Der(R))\geq \dim(R)$. We have $n$ natural derivations to work with in our setup, so we'd prefer to work with rings of dimension at most $n$. We therefore tensor with $\Fra(A)$ to reduce the dimension to something that (WJ) can apply to. We now check that the rings of interest satisfy our hypothesis.

\begin{prop}\label{WJex}
Let $A=\Ar$ and $R=\ATrn$, then Hypothesis \ref{exhyp} is satisfied.
\end{prop}

\begin{proof}
Everything but the weak Jacobian condition has already been checked or is clear. Choose $\ip\in \Spec(R\otimes_A\Fra(A))$, let the height of $\ip$ be $h$. Let $\ip'\in\Spec(R)$ be the contraction of $\ip$, it also has height $h$. Let $\iq\in\Spec(R)$ be a maximal ideal containing $\ip'$, let $Q=R_\iq/\ip'$, this is a local ring of dimension $n+1-q$. 

Let $\Delta$ be the derivations of $R$ induced by derivations of $\Der(R\otimes_A K))$, then $\Delta$ is generated by the $n$ elements $\frac{\partial}{\partial T_i}$. The derivations in $\Delta$ are exactly the $A$-derivations of $R$. We have $\rank J(\ip;\Der(R\otimes_A K))=\rank J(\ip'; \Delta)\leq h$ by Lemma \ref{RegQuot}; we must show that we have equality. Following the proof of \cite[Theorem 100]{Mat1}, we see that $\rank \Der_A(Q)=n-\rank J(\ip';\Delta)$, so it is enough to show that $\rank \Der_A(Q)=n-h$. 

To do this, we largely follow \cite[Theorem 98]{Mat1}. By Theorem \ref{rnull}, we have $\iq\cap A=(x_0)$ for a principal maximal ideal $(x_0)$. We extend $x_0$ to a system of parameters $x_0,\dots,x_r$ of $Q$, here $r=n-h$. By the Cohen structure theorem, we have that $\widehat Q$ is an integral extension of $F[[x_0,\dots,x_r]]$ where $F=R/\iq$. Writing $F=(R/(x_0))/(\iq/(x_0)),$ we see that it is the quotient of an affinoid algebra over the field $A/(x_0)$, so it is an integral extension. 

Take any $A$-linear derivation $D$ of $Q$ vanishing on $x_1,\dots,x_r$ and extend it to $\widehat Q$. Then $D$ vanishes on $x_0$ as $x_0\in A$, and it vanishes on $F$ as it vanishes on $A/(x_0)$ and $F$ is an integral extension of $A/(x_0)$. So $D$ vanishes on all of $F[[x_0,\dots,x_r]]$, so it must also vanish on $\widehat Q$ and therefore $Q$. Indeed, given any $y\in Q$, there is some integral relation $f(T)$ for $y$ over $F[[x_0,\dots,x_n]]$ of minimal degree.  Then $0=D(f(y))=f'(y)D(y)$ and $f'(y)$ is nonzero as the characteristic is 0, so $D(y)=0$. So $D$ is determined by the tuple $(D(x_1),\dots,D(x_m))$, and so $\rank \Der_A(Q)\leq n-h$. As $\rank J(\ip'; \Delta)\leq h$, we must have equalities in both equations as desired.
\end{proof}

\begin{remark}
Keeping notation as in Proposition \ref{WJex}, we note we could apply \cite[Theorem 100]{Mat1} to get (WJ) for $R\otimes_A\Fra(A)$ if we knew that every maximal ideal $\imax\subset R\otimes_A\Fra(A)$ has residue field algebraic over $\Fra(A)$. While this seems plausible, we were unable to prove it directly. It is tempting to try to prove a version of Weierstrass preparation and division for $R\otimes_A\Fra(A)$, but following the standard proof for affinoid algebras over a field doesn't quite work.
\end{remark}

\begin{prop}\label{BJ2}
Any ring $R$ satisfying Hypothesis \ref{exhyp} is J-2.
\end{prop}

\begin{proof}
We proceed by induction on the dimension of $A$. As $A$ is an integral domain, the base case is when $A$ is a field. By Proposition \ref{WJex}, $R$ is (WJ) in this case, so by Proposition \ref{WJisJ2} it is J-2. For the inductive step, by Corollary \ref{J2quot} it is again enough to show that for every $\iq\in \Spec(R)$, the set $\Reg(R/\iq)$ is open in $\Spec(R/\iq)$. 

Let $\iq\cap A=Q$, then we can reduce to the case where $Q=\{0\}$ by noting that $R/\iq \cong (R/Q)/(\iq/Q)$ and that $(A/Q)[T_1,\dots,T_n]\subset R/Q\subset (A/Q)[[ T_1,\dots,T_n]]$. In this case, $\iq$ is in the image of the injection $\Spec(R\otimes_A \Fra(A))\hookrightarrow \Spec(R)$, so we can work with the preimage $\iq'$. By Hypothesis \ref{exhyp}, (WJ) holds in $R\otimes_A \Fra(A)$, so Proposition \ref{WJisJ2} implies that this ring is J-2.

In particular, letting $\ip=\iq'$ in Proposition \ref{WJisJ2} we get a nonempty open set $D(\overline{gh})$ containing only regular primes. The construction of $\overline{gh}$ gives the required generators $f_1,\dots,f_r\in \ip\otimes_A \Fra(A)$ and derivations $D_1,\dots, D_r\in \Der(R\otimes_A\Fra(A))$ for Lemma \ref{RegQuot} to apply. Finding a common denominator in $\Fra(A)$ gives an element $d\in A$ such that the entire argument works in $R_d$, so the set $\Reg((R/\iq)_d)$ is open.

To complete the proof, we just need to show that $$\Reg(R\cap V(d)\subset\Spec(R\}/\iq)\cap V(d)$$ is open. This is equivalent to checking that $\Reg((R/d)/(\iq, d))$ is open in $\Spec(R/d)/(\iq, d))$. We just need to check this for each of the finitely many minimal primes. As $R/(d)$ satisfies Hypothesis \ref{exhyp} for $A=A/(d)$ and $\dim(A/d)<\dim(A)$ these follow from the inductive hypothesis.

\end{proof}

\begin{prop}\label{BGring}
Any ring $R$ satisfying Hypothesis \ref{exhyp} is a G-ring.
\end{prop}

\begin{proof}
Here we are adapting \cite[Theorem 101]{Mat1}. By \cite[Theorem 75]{Mat1}, it is enough to show that for every maximal ideal $\imax$ of $R$, the local ring $R_\imax$ has geometrically regular formal fibers. As $\Q\subset R$, it is enough to check that the formal fibers are regular; the argument is the same as in Corollary \ref{J2quot}. Concretely, it is enough to show that for every prime $\ip\in\Spec(\widehat{R_\imax})$, the local ring $(\widehat{R_\imax})_\ip/\ip$ is regular. As in Proposition \ref{BJ2}, we can reduce to the case where $\ip\cap A=\{0\}$ by replacing $A$ with $A/(\ip\cap A)$.

When $\ip\cap A=(0)$, we look at the image $\ip'$ of $\ip$ in $R\otimes_{A} \Fra(A)$. Here (WJ) holds by Proposition \ref{WJex}, so we get derivations $D'_1,\dots, D'_r$ and $f'_1,\dots, f'_r\in \ip\otimes_{A} \Fra(A)$ such that $\rank J(f'_1,\dots, f'_r; D'_1,\dots, D'_r)(\ip')=r$. We can multiply by an element of $A$ to clear the denominators of the matrix and restrict the derivations to $R$. This gives $f_1,\dots,f_r\in\ip$ with $\rank J(f_1,\dots,f_r; D_1,\dots,D_r)(\ip)=r$. 

We can extend the derivations to $\widehat{R_\imax}$ and view the $f_i$ as elements of $\ip (\widehat{R_\imax})_P$ to get $\rank J(f_1,\dots,f_r; D_1,\dots,D_r)(\ip (\widehat{R_\imax})_P)=r$. By \cite[Theorem 19]{Mat1}, we have $\hgt\ip \widehat{R_\imax}=\hgt\ip=r$ so we can again apply \ref{RegQuot} to see that $(\widehat{R_\imax})_P/\ip$ is regular as desired.
\end{proof}

Combining these gives the desired theorem.

\begin{theorem}\label{Rexc}
Any ring $R$ satisfying Hypothesis \ref{exhyp} is excellent.
\end{theorem}

\begin{proof}
This follows from \ref{Breg}, \ref{BJ2}, and \ref{BGring}.
\end{proof}

\begin{corollary}\label{Bexc}
The rings $\ATrn$ and $\Bin$ are excellent. As excellence is stable under passage to finitely generated algebras, this implies that the rings $C$ of Hypothesis \ref{etalehyp} arising from \'etale morphisms are also excellent. 
\end{corollary}

\begin{corollary}\label{FFstalks}
The stalks of the adic Fargues-Fontaine curve are noetherian.
\end{corollary}

\begin{proof}
Temkin proved this for rigid analytic spaces, making essential use of the fact that affinoid algebras are excellent. His proof works just as well for the Fargues-Fontaine curve now that we have proven that the extended Robba rings are excellent. We give a brief sketch of the proof, a more detailed version is given in Proposition \cite[15.1.1]{Con}.

The local ring $\mathcal{O}_x$ at a point $x$ of the Fargues-Fontaine curve can be written the direct limit of a directed system $(A_i)$ of rational domains $\Spa(A_i, A_i^+)$ containing $x$ where the $A_i$ are all extended Robba rings. Let $\imax$ denote the maximal ideal of the local ring $\mathcal{O}_x$, and let $\imax_i\in \Spec(A_i)$ be the image of $\imax$ in the map $\Spec(A)\ra\Spec(A_i)$ coming from the direct limit. Letting $B_i=(A_i)_{\imax_i}$, the directed system of the $B_i$ also has limit $\mathcal{O}_x$.

Huber showed that the transition maps $A_i\ra A_j$ are flat \cite[II.1.iv]{Hub0}. The directed system $(B_i)$ therefore consists of local noetherian rings with flat local transition maps, it is shown in EGA that the limit is noetherian if for sufficiently large $i$, we have $\imax_i B_j=\imax_j$ for all $j\geq i$. As the transition maps are flat, we have $$\dim(B_j)=\dim(B_i)+\dim(B_j/\imax_i B_j)\geq \dim(B_i)$$ for $j\geq i$. As the dimension of the $B_j$ is bounded above, we must have some $i_0$ such that $\dim(B_i)=\dim(B_{i_0})$ for all $i\geq i_0$ and so $\dim(B_i/\imax_{i_0} B_i)=0$ for $i\geq i_0$. 

So it is enough to show that $B_i/\imax_{i_0}B_i$ must be reduced. This is the localization of a fiber algebra of a map of excellent rings, so it is reduced by the argument in \cite{Con}.
\end{proof}

\end{document}